\documentclass[reqno]{amsart}
\usepackage{hyperref}
\usepackage{amsmath}
\usepackage{amssymb}
\begin{document}
\title[The fifth-order KPII equation]
{Global well-posedness for the fifth-order Kadomtsev-Petviashvili II equation in anisotropic Gevrey spaces}

\author[A. Boukarou, D. O. da Silva, K. Guerbati, Kh. Zennir]
{A. Boukarou, D. O. da Silva, K. Guerbati and Kh. Zennir}
\address{Aissa Boukarou \newline
	Laboratoire de Math\'ematiques et Sciences appliqu\'ees Universit\'e de Ghardaia, Algeria}
\email{boukarou.aissa@univ-ghardaia.dz}
\address{Daniel Oliveira da Silva \newline
	Department of Mathematics, Nazarbayev University, Nur-Sultan, Kazakhstan}
\email{daniel.dasilva@nu.edu.kz}
\address{Kaddour Guerbati \newline
	Laboratoire de Math\'ematiques et Sciences appliqu\'ees Universit\'e de Ghardaia, Algeria}
\email{guerbati\_k@yahoo.com}
\address{Khaled Zennir \newline
Department of Mathematics, College of
Sciences and Arts, Qassim University, Ar-Rass, Kingdom of Saudi
Arabia}
\email{k.zennir@qu.edu.sa}

\subjclass[2010]{35Q35, 35Q53}
\keywords{KPII equation, Gevrey space, radius of spatial analyticity}
\begin{abstract}
We show that the fifth-order Kadomtsev-Petviashvili II equation is globally well-posed in an anisotropic Gevrey space, which complements earlier results on the well-posedness of this equation in anisotropic Sobolev spaces.
\end{abstract}
\maketitle
\numberwithin{equation}{section}
 \newtheorem{thm}{Theorem}[section]
 \newtheorem{cor}[thm]{Corollary}
 \newtheorem{lem}[thm]{Lemma}
 \newtheorem{prop}[thm]{Proposition}
 \theoremstyle{definition}
 \newtheorem{defn}[thm]{Definition}
 \theoremstyle{remark}
 \newtheorem{rem}[thm]{Remark}
 \newtheorem*{ex}{Example}
 \numberwithin{equation}{section}
\allowdisplaybreaks
\section{Introduction}
 The fifth-order Kadomtsev-Petviashvili II (KPII) equation is the partial differential equation
 \begin{equation} 
 \left\{
 \begin{array}{ll}\label{p1}
 \partial_{t}u - \partial_{x}^{5}u + \partial^{-1}_{x} \partial^{2}_{y} u + u \partial_{x} u = 0, \\
 u(x,y,0) = f(x,y), 
 \end{array}
 \right.
 \end{equation}
 where $u = u(x,y,t)$ and $(x,y,t) \in \mathbb{R}^{3}$.  This equation belongs to the class of KP equations, which are models for the propagation of long dispersive nonlinear waves which are essentially unidirectional and have weak transverse effects.  Due to the asymmetric nature of the equation with respect to the spatial derivatives, it is natural to consider the Cauchy problem for \eqref{p1} with initial data in the anisotropic Sobolev spaces $H^{s_1, s_2} (\mathbb{R}^{2})$, defined by the norm
 \[
 \| u \|_{H^{s_1, s_2}(\mathbb{R}^{2})} = \left( \int \langle \xi \rangle^{2 s_1} \langle \eta \rangle^{2 s_2} | \tilde{u}(\xi, \eta) |^{2} d\xi d\eta \right)^{1/2},
 \]
 where $\tilde{f}$ denotes the spatial Fourier transform
\[
\tilde{f}(\xi, \eta, t) = \int_{\mathbb{R}^{2}} f(x,y,t) e^{-i (x \xi + y \eta)}\ dx dy.
\]
For initial data in these spaces, Saut and Tzvetkov \cite{ST2000} proved that the problem \eqref{p1} is locally well-posed for initial data $f$ in $H^{0,0}(\mathbb{R}^{2}) = L^{2}(\mathbb{R}^{2})$.  This result was improved in \cite{ILM2006} by Isaza, L\'opez and Mej\'{i}a, who reduced the minimal regularity for initial data to $s_{1} > -5/4$ and $s_{2} \geq 0$.  They also showed that the problem is globally well-posed in $H^{s_{1},0}(\mathbb{R}^{2})$ with $s_{1} > -4/7$.
 
In the present work, we will consider the Cauchy problem for equation \eqref{p1} with initial data in an anisotropic Gevrey space $G^{\sigma_1, \sigma_2}(\mathbb{R}^{2})$, which we define as the completion of the Schwartz functions with respect to the norm
\[
\| f \|_{G^{\sigma_1, \sigma_2}(\mathbb{R}^{2})} = \left(\int_{\mathbb{R}^{2}} e^{2 \sigma_1 |\xi|} e^{2 \sigma_2 |\eta|} | \tilde{f}(\xi,\eta) |^{2} d\xi d\eta\right)^{1/2}.
\]
The primary reason for considering initial data in these spaces is because of the following theorem:
\begin{prop}[Paley-Wiener Theorem]\label{wiener}
Let $\sigma > 0$, and suppose $f \in L^{2}(\mathbb{R})$.  Then the following are equivalent:
\begin{enumerate}
\item The function $f$ is the restriction to the real line of a function $F$ which is holomorphic in the strip 
$$ S_{\sigma} = \lbrace x+iy \in \mathbb{C}:\  | y | < \sigma \rbrace,$$ and satisfies
\[
\sup_{| y | < \sigma} \| F(x+iy) \|_{L^{2}}< \infty.
\]
\item $e^{\sigma |\xi|} \hat{f}(\xi) \in L^{2}_{\xi}(\mathbb{R})$.
\end{enumerate}
\end{prop}

A proof of this can be found in chapter 4, section 7 of \cite{K1976}.  The quantity $\sigma$ is known as the \emph{radius of analyticity}.  Thus, for functions in $G^{\sigma_1, \sigma_2}$, if we hold one variable fixed, the resulting function in the other variable will have a holomorphic extension satisfying the stated bounds.

In addition to the holomorphic extension property, Gevrey spaces satisfy the embeddings $G^{\sigma_1, \sigma_2} \hookrightarrow G^{\sigma'_1, \sigma_2'}$ for $\sigma_{i}' < \sigma_{i}$, which follow from the corresponding estimates
\begin{equation}\label{embed}
\| f \|_{G^{\sigma_1', \sigma_2'}} \lesssim \| f \|_{G^{\sigma_1, \sigma_2}}.
\end{equation}
With these facts in mind, we may state our main results for this paper.  The first result relates to the short-term persistence of analyticity of solutions.
\begin{thm} \label{the1.2}
Let  $\sigma_1 \geq 0$ and $\sigma_2 \geq 0$. Then for all initial data $f \in  G^{\sigma_1, \sigma_2}$, there exists $\delta = \delta(\parallel f\parallel_{G^{\sigma_1, \sigma_2}}) > 0$ and a unique solution $u$ of \eqref{p1} on the time interval $[0, \delta]$ such that $$u \in C\left([0, \delta]; G^{\sigma_1, \sigma_2} (\mathbb{R}^{2})\right).$$  Moreover the solution depends continuously on the data $f$.  In particular, the time of existence can be chosen to satisfy
\[
\delta = \frac{c_{0}}{(1 + \| f \|_{G^{\sigma_1, \sigma_2}})^{\gamma}},
\]
for some constants $c_{0}> 0$ and $\gamma > 1$.  Moreover, the solution $u$ satisfies
\[
\sup_{t \in [0,\delta]}\| u(t) \|_{G^{\sigma_1, \sigma_2}} \leq 2 \| f \|_{G^{\sigma_1, \sigma_2}}.
\]
\end{thm}

Our second main result concerns the evolution of the radius of analyticity for the $x$-direction.
\begin{thm} \label{the1.3}
Let  $\sigma_1 > 0$ and $\sigma_2 \geq 0$, and assume $f \in  G^{\sigma_1, \sigma_2}$.  Then the solution $u$ given by Theorem \ref{the1.2} extends globally in time, and for any $T > 0$, we have
$$
u \in C\left([0, T], G^{\sigma(T), 0}(\mathbb{R}^{2}) \right) \quad \text{with} \quad \sigma(T) = \min\left\lbrace \sigma_1, C T^{-1}\right\rbrace,$$
where $C > 0$ is a constant which does not depend on $T$. 
\end{thm}
\noindent The method used here for proving lower bounds on the radius of analyticity was introduced in \cite{ST2015} in the study of the 1D Dirac-Klein-Gordon equations. It was applied to the modified Kawahara equation \cite{PD2019} and the non-periodic KdV equation in \cite{SD2017}, to the dispersion-generalized periodic KdV equation in \cite{HKS2017}, the Ostrovsky equation \cite{BZGG2020a}, and to the quartic generalized KdV equation on the line in \cite{ST2017}.

The rest of the paper is organized as follows: in section \ref{prel}, we introduce the various tools which will be used in the proofs of our main theorems.  We then prove Theorems \ref{the1.2} and \ref{the1.3} in sections \ref{local} and \ref{global}, respectively.
 
%%%%%%%%%%%%%%%%%%%%%%%%%%%%%%%%%%%%%%%%%%%%%%%%%%%%%%% 
\section{Preliminaries}\label{prel}
%%%%%%%%%%%%%%%%%%%%%%%%%%%%%%%%%%%%%%%%%%%%%%%%%%%%%%% 

In this section, we introduce the preliminary results necessary for our proofs.  We begin by fixing the notation to be used.  We begin by noting that, in addition to the spatial Fourier transform  stated above, we will also need the full spacetime Fourier transform, which we denote by
\[
\hat{f}(\xi, \eta, \tau) = \int_{\mathbb{R}^{3}} f(x,y,t) e^{-i (x \xi + y \eta + t \tau)}\ dx dy dt.
\]
In both cases, we will denote the corresponding inverse transform of a function $f = f(\xi, \eta)$ or $f = f(\xi, \eta, \tau)$ by $\mathfrak{F}^{-1}(f)$.

To simplify the notation, we introduce some operators which will be introduced later.  We first introduce the operator $A^{\sigma_1, \sigma_2}$, which we define as
\begin{equation}\label{1.1.3}
    A^{\sigma_1, \sigma_2}f = \mathfrak{F}^{-1}\left( e^{\sigma_1 |\xi|} e^{\sigma_2 |\eta|} \hat{f} \right).
\end{equation}
With this, we may then define another useful operator
\begin{equation}\label{nu}
    N(f) = \partial_{x}\left[(A^{\sigma_1, \sigma_2}f)^2 - A^{\sigma_1, \sigma_2}(f^2)\right].
\end{equation}
For $x \in \mathbb{R}^{n}$, we denote $\langle x \rangle = (1 + |x|^{2})^{1/2}$.  Finally, we write $a \lesssim b$ if there exists a constant $C > 0$ such that $a \leq Cb$, and $a \sim b$ if $a \lesssim b \lesssim a$.  If the constant $C$ depends on a quantity $q$, we denote this by $a \lesssim_{q} b$.

Since our proofs will rely heavily on the theory developed by Isaza, L\'{o}pez and Mej\'{i}a, let us state the function spaces they used explicitly, so that we can state their useful properties which we will exploit in our modifications of their spaces.  The main function spaces they used are the so-called Bourgain spaces, adapted to the KP II equation, whose norm is given by
\[
\| u \|_{X^{s_1, s_2, b, \varepsilon}} =
\left(\displaystyle \int_{\mathbb{R}^{3}} \lambda^{2}(s_1, s_2, b, \varepsilon) |\hat{u}(\xi, \eta, \tau)|^{2} d\xi d\eta d\tau \displaystyle\right) ^{\frac{1}{2}},
\]
where 
\[
\lambda(s_1, s_2, b, \varepsilon) = \langle \xi \rangle^{s_1} \langle \eta \rangle^{s_2} \langle \tau - m(\xi, \eta) \rangle^{b} \left\langle \frac{\tau - m(\xi,\eta)}{1 + |\xi|^{5}}\right\rangle^{\varepsilon}.
\]
with $m(\xi, \eta)=\xi^{5}-\frac{\eta^{2}}{\xi}$.  As is well-known (see, for example, Corollary 2.10 of \cite{T2006}), these spaces satisfy the embedding $X^{s_1, s_2, b, \varepsilon} \hookrightarrow C \left(\mathbb{R}; H^{s_1, s_2} (\mathbb{R}^{2})\right)$.  Thus, solutions constructed in $X^{s_1, s_2, b, \varepsilon}$ belong to the natural solution space.

When considering local solutions, it is useful to consider localized versions of these spaces.  For a time interval $I$ and a Banach space $Z$, we define the localized space $Z(I)$ by the norm
\[
\| u \|_{Z(I)} = \inf \{ \| v \|_{Z}:\ v = u \textrm{ on } I \}.
\]
It is easy to see that estimates which hold for the original spaces also hold for the localized spaces.  For simplicity, we will omit the interval from the notation when there is no chance of confusion.

The final preliminary fact we must state is the following bilinear estimate, which is Lemma 1.1 of \cite{ILM2006}:
\begin{lem}\label{bilinear}
Let $s_1 > - 5/4$, $s_2 \geq 0$, $b > 1/2$, and define $s = \max \{0, -s_1\}$.  If $\varepsilon$ and $\beta$ satisfy the inequalities
\[
0 \leq \varepsilon \leq \min \left\{ \frac{2}{5}\left( \frac{5}{4} - s \right), \frac{3}{20} \right\}
\]
and
\[
\max \left\{ \frac{9}{20}, \frac{1}{2} - \frac{1}{2}\left( \frac{5}{4} - s \right) + \varepsilon \right\} \leq \beta < \frac{1}{2},
\]
then
\[
\| \partial_{x} (uv) \|_{X^{s_1, s_2, -\beta, \varepsilon}} \lesssim \| u \|_{X^{s_1, s_2, b, \varepsilon}} \| v \|_{X^{s_1, s_2, b, \varepsilon}}.
\]
\end{lem}

%%%%%%%%%%%%%%%%%%%%%%%%%%%%%%%%%%%%%%%%%%%%%%%%%%%%%%% 
\section{Proof of Theorem \ref{the1.2}}\label{local}
%%%%%%%%%%%%%%%%%%%%%%%%%%%%%%%%%%%%%%%%%%%%%%%%%%%%%%%

We may now begin the proof of Theorem \ref{the1.2}.  Since this is a modification of the proof of Isaza, L\'{o}pez, and Mej\'{i}a in \cite{ILM2006}, we will merely outline the essential steps.  To begin, consider the linear problem
\begin{align*}%\begin{equation}\label{p2}\begin{aligned}
& \partial_{t} u - \partial_{x}^{5} u + \partial^{-1}_{x} \partial^{2}_{y} u = F, \\
& u(0) = f.
\end{align*}%\end{aligned}\end{equation}
By Duhamel's principle the solution can be written as
\begin{equation}\label{12.11}
u(t) = S(t)f - \dfrac{1}{2} \int_{0}^{t} S(t-t') F(t') dt',
\end{equation}
where $$\widetilde{S(t)f}(\xi,\eta) = e^{i t m(\xi,\eta)} \tilde{f}(\xi,\eta).$$
Let $\psi \in C^{\infty}_{0}(\mathbb{R})$ be supported in the interval $(-2,2)$ such that $0 \leq \psi(t) \leq 1$ and $\psi = 1$ on $[-1,1]$, and let $\phi_{\delta} \in C^{\infty}_{0}(\mathbb{R})$ be supported on $(-2 T^{-1/2}, 2 T^{-1/2})$ such that $0 \leq \phi_{\delta}(t) \leq 1$ and $\phi_{\delta} = 1$ on $[-T^{-1/2},T^{-1/2}]$.  We then observe that we may decompose the integral operator on the right-hand side of  equation \eqref{12.11} in the form
\begin{equation}\label{112}
-\dfrac{1}{2}\int_{0}^{t}S(t-t')f(t')dt'=I_{\phi_{\delta}}(f)+II_{\phi_{\delta}}(f)+III_{\phi _{\delta}}(f),
\end{equation}
where 
$$I_{\phi_{\delta}}(f) = C \int e^{i(x\xi+y\eta)} e^{itm(\xi,\eta)} \dfrac{e^{itp(\xi, \eta, \tau)} - 1}{i p(\xi, \eta, \tau)} \phi_{\delta}(p(\xi, \eta, \tau)) \hat{f}(\xi,\eta,\tau)\ d\xi d\eta d\tau,  $$

$$II_{\phi_{\delta}}(f) = C \int e^{i(x\xi+y\eta)} \dfrac{e^{it\tau}}{ip(\xi, \eta, \tau)}(1-\phi_{\delta}(p(\xi, \eta, \tau))) \hat{f}(\xi,\eta,\tau)\ d\xi d\eta d\tau,  $$

$$III_{\phi_{\delta}}(f) = -C \int e^{i(x\xi+y\eta)} e^{itm(\xi,\eta)} \dfrac{(1-\phi_{\delta}(p(\xi, \eta, \tau))}{i p(\xi, \eta, \tau)} \hat{f}(\xi,\eta,\tau)\ d\xi d\eta d\tau.  $$
Here, $p(\xi, \eta, \tau) = \tau - m(\xi,\eta)$.  We then define a modification $G_{\delta}(f)$ of the integral operator in  (\ref{112}),
$$G_{\delta}(f)= \psi(t/\delta)I_{\phi_{\delta}}(f)+II_{\phi_{\delta}}(f)+\psi(t/\delta)III_{\pi_{\delta}}(f). $$
A simple computation will show that if $0 < \delta < 1$ and $t \in [-\delta, \delta]$, then
\[
G_{\delta}(f)(t) = -\dfrac{1}{2}\int_{0}^{t}S(t-t')f(t')dt'.
\]

Now define a sequence $\{ u_{n} \}_{n = 0}^{\infty}$ of functions which are solutions to the equations
\begin{align*}
    & \partial_{t} u_{0} - \partial_{x}^{5} u_{0} + \partial_{x}^{-1} \partial_{y}^{2} u_{0} = 0 & \qquad & \partial_{t} u_{n} - \partial_{x}^{5} u_{n} + \partial_{x}^{-1} \partial_{y}^{2} u_{n} = - u_{n-1}\partial_{x} u_{n-1} \\
    & u_{0}(x, y, 0) = f(x, y) & & u_{n}(x, y, 0) = f(x, y)
\end{align*}
By the discussion above, for $t \in (-\delta, \delta)$, we have the identity $u_{n}(x, y, t) = \Phi(u_{n-1}(x, y, t))$, where $$\Phi(u) = \psi(t) S(t) f + G_{\delta}\left( \partial_{x}(u^2) \right).$$  To this operator, we apply the following estimate, which follows from equations (1.2) and (1.4) of \cite{ILM2006} and the commutativity of Fourier multipliers:
\begin{lem}\label{lem2.2}
Let $\sigma_1 \geq 0$, $\sigma_2 \geq 0$, $\frac{1}{2} < b < 1$, $\beta \in (0, 1-b)$, and $0 < \delta < 1$.  Then
\[
\| \Phi(u) \|_{Y^{\sigma_1, \sigma_2, b}} \leq \| f \|_{G^{\sigma_1, \sigma_2}} + C \delta^{\gamma} \| \partial_{x} (u^2) \|_{Y^{\sigma_1, \sigma_2, -\beta}}
\]
for some $0 < \gamma < 1$.
\end{lem}
\noindent To this result, we apply the following lemma, which is a corollary of Lemma \ref{bilinear}:
\begin{lem}\label{goodest}
For $\sigma_1 \geq 0$, $\sigma_2 \geq 0$, and $b > 1/2$, we have
\[
\| \partial_{x}(uv) \|_{Y^{\sigma_1, \sigma_2, 9/20}} \lesssim \| u \|_{Y^{\sigma_1, \sigma_2, b}} \| v \|_{Y^{\sigma_1, \sigma_2, b}}.
\]
\end{lem}
\begin{proof}
From the triangle inequality, it is easy to see that
\begin{align*}
  & e^{2 (\sigma_1 |\xi| + \sigma_2 |\eta|)} \left| \widehat{uv}(\xi,\eta, \tau) \right|^{2} \\
  & \quad = e^{2 (\sigma_1 |\xi| + \sigma_2 |\eta|)} \left| \int \hat{u}(\xi - \xi_1, \eta - \eta_1, \tau - \tau_1) \hat{v}(\xi_1, \eta_1, \tau_1)\ d\xi d\eta d\tau \right|^{2} \\
  & \quad \leq \left| \int e^{\sigma_1 |\xi - \xi_1| + \sigma_2 |\eta - \eta_1|}\hat{u}(\xi - \xi_1, \eta - \eta_1, \tau - \tau_1) \right. \times \\
  & \qquad \qquad \left. \times e^{\sigma_1 |\xi_1| + \sigma_2 |\eta_1|} \hat{v}(\xi_1, \eta_1, \tau_1)\ d\xi d\eta d\tau \right|^{2} \\
  & \quad = \left| \widehat{A^{\sigma_1, \sigma_2}u A^{\sigma_1, \sigma_2}v} \right|^{2}.
\end{align*}
It follows that
\begin{align*}
\| \partial_{x}(uv) \|_{Y^{\sigma_1, \sigma_2, 9/20}} & \leq \| \partial_{x}(A^{\sigma_1, \sigma_2}u A^{\sigma_1, \sigma_2}v) \|_{X^{0,0,\frac{9}{20},0}} \\
& \lesssim \| A^{\sigma_1, \sigma_2}u \|_{X^{0, 0, b, 0}} \| A^{\sigma_1, \sigma_2}v \|_{X^{0, 0, b, 0}} \\
& = \| u \|_{Y^{\sigma_1, \sigma_2, b}} \| v \|_{Y^{\sigma_1, \sigma_2, b}}
\end{align*}
by Lemma \ref{bilinear}.
\end{proof}
With these results, it is a simple matter to show that
\[
\| u_{n} \|_{Y^{\sigma_1, \sigma_2, b}} \leq \| f \|_{G^{\sigma_1, \sigma_2}} + C \delta^{\gamma} \| u_{n-1} \|_{Y^{\sigma_1, \sigma_2, b}}^{2}.
\]
Using a simple proof by induction, one may show that
\[
\| u_{n} \|_{Y^{\sigma_1, \sigma_2, b}} \leq 2 \| f \|_{G^{\sigma_1, \sigma_2}}
\]
for all $n \in \mathbb{N} \cup \{ 0 \}$, if we choose $\delta$ such that
\begin{equation}\label{prop1}
\delta < \frac{1}{(C \| f \|_{G^{\sigma_1, \sigma_2}})^{1/\gamma}}.
\end{equation}
The final step is to show that the sequence converges.  Applying Lemmas \ref{lem2.2} and \ref{goodest} once again, a similar computation will show that
\begin{align*}
\| u_{n} - u_{n-1} \|_{Y^{\sigma_1, \sigma_2, b}} & = \| \Phi(u_{n}) - \Phi(u_{n-1}) \|_{Y^{\sigma_1, \sigma_2, b}} \\
& \leq C \delta^{\gamma} M_{n-1} \| u_{n-1} - u_{n-2} \|_{Y^{\sigma_1, \sigma_2, b}} \\
& \leq 4 C \delta^{\gamma} \| f \|_{G^{\sigma_1, \sigma_2}} \| u_{n-1} - u_{n-2} \|_{Y^{\sigma_1, \sigma_2, b}}.
\end{align*}
where
\[
M_{n-1} = \| u_{n-1} \|_{Y^{\sigma_1, \sigma_2, b}} + \| u_{n-2} \|_{Y^{\sigma_1, \sigma_2, b}}.
\]
Thus, the sequence will converge if $\delta$ satisfies 
\begin{equation}\label{prop2}
\delta < \frac{1}{(4 C \| f \|_{G^{\sigma_1, \sigma_2}})^{1/\gamma}}.
\end{equation}
sufficiently small.  Thus, the sequence converges to a solution $u \in Y^{\sigma_1, \sigma_2, b} \subset C([0,\delta]; G^{\sigma_1, \sigma_2})$.  

To show uniqueness, suppose $u$ and $v$ are solutions to \eqref{p1}, and let $w = u - v$.  Then $w$ satisfies the equation
\[
\partial_{t}w - \partial_{x}^{5}w + \partial_{x}^{-1} \partial_{y}^{2} w = v \partial_{x}v - u \partial_{x}u,
\]
with $w(x, y, 0) = 0$.  Multplying by $w$ and integrating in $x$ and $y$ yields
\[
\frac{1}{2} \frac{d}{dt} \| w(t) \|_{L^{2}_{x,y}} \leq \left( \| u_{x}(t) \|_{L^{\infty}_{x,y}} + \| v_{x}(t) \|_{L^{\infty}_{x,y}} \right) \| w(t) \|_{L^{2}_{x,y}}.
\]
By the embedding
\[
\| u_{x}(t) \|_{L^{\infty}_{x,y}} \lesssim \| u(t) \|_{G^{\sigma_1, \sigma_2}} \leq 2 \| f \|_{G^{\sigma_1, \sigma_2}},
\]
which also holds for $v$, it will follow from Gr\"{o}nwall's Inequality that $w = 0$.  This completes the proof of Theorem \ref{the1.2}.

As a concluding remark, we observe that for later convenience, we may choose the time of existence to be
\[
\delta = \frac{C}{(1 + \| f \|_{G^{\sigma_1, \sigma_2}})^{1/\gamma}}.
\]
For appropriate choice of $C$, this will satisfy inequalities \eqref{prop1} and \eqref{prop2}.

%%%%%%%%%%%%%%%%%%%%%%%%%%%%%%%%%%%%%%%%%%%%%%
\section{Proof of Theorem \ref{the1.3}}\label{global}
%%%%%%%%%%%%%%%%%%%%%%%%%%%%%%%%%%%%%%%%%%%%%%

In this section, we begin the proof of Theorem \ref{the1.3}.  The first step is to obtain estimates on the growth of the norm of the solutions.  For this, we will need the following approximate conservation law:
\begin{prop}\label{th03}
Let $\sigma_1 \geq 0$.  Then there is a $b \in (1/2, 1) $ and a $C > 0$, such that $u \in Y^{\sigma_1, 0, b}(I)$ is a solution to the Cauchy problem (\ref{p1}) on the time interval $[0, \delta]$, we have the estimate
\begin{equation}\label{3.3.1}
\sup_{ t\in[0, T]} \| u(t) \|^{2}_{G^{\sigma_1, 0}} \leq \| f \|^{2}_{G^{\sigma, 0}} + C \sigma \| u \|^{3}_{Y^{\sigma_1, 0, b}(I)}.
\end{equation}
% Moreover, we have
%\begin{equation}\label{3.3.2}
%\sup_{t\in[0, T]} | u(t) |^{2}_{G^{\sigma, 0, 0, 0}} \leq | u(0) |^{2}_{G^{\sigma, 0, 0, 0}} + C \delta | u(0) |^{3}_{G^{\sigma, 0, 0, 0}}.
%\end{equation}
\end{prop}
\noindent Before we may state the proof, let us first state some preliminary lemmas.  The first is an immediate consequence of Lemma 12 in \cite{SD2017} and the comments immediately following it:
\begin{lem}\label{200}
For $\sigma > 0$ and $\xi, \xi_{1} \in \mathbb{R}$, we have
\begin{align*}
e^{\sigma |\xi - \xi_1|} e^{\sigma |\xi_1|} - e^{\sigma |\xi|} \lesssim \sigma \frac{\langle \xi - \xi_1 \rangle \langle \xi_1 \rangle}{\langle \xi \rangle} e^{\sigma |\xi - \xi_1|} e^{\sigma |\xi_1|}.
\end{align*}
\end{lem}
\noindent This will be used to prove the following key estimate:
\begin{lem}\label{lem3.3}
Let $N(u)$ be as in equation \eqref{nu} for $\sigma_1 \geq 0$ and $\sigma_2 = 0$.  Then for $b$ and $\beta$ as in Lemma \ref{bilinear}, we have
\[
\| N(u) \|_{Y^{0, 0, -\beta}} \leq C \sigma \| u \|^{2}_{Y^{\sigma_1, 0, b}}.
\]
\end{lem}
\begin{proof}
We first observe that the inequality in Lemma \ref{bilinear}, for the case $\varepsilon = 0$, is equivalent to
\begin{align*}
& \left\| \xi \lambda(s_1, s_2, -\beta, 0) \int \frac{\hat{f}(\xi - \xi_1, \eta - \eta_1, \tau - \tau_1)}{\langle \xi - \xi_1 \rangle^{s_1} \langle \eta - \eta_1 \rangle^{s_2} \langle \phi(\xi - \xi_1, \eta - \eta_1, \tau - \tau_1) \rangle^{b}} \right. \times \\
& \times \left. \frac{\hat{g}(\xi_1, \eta_1, \tau_1)}{\langle \xi_1 \rangle^{s_1} \langle \eta_1 \rangle^{s_2} \langle \phi(\xi_1, \eta_1, \tau_1) \rangle^{b}} \ d\xi_1 d\eta_1 d\tau_1 \right\|_{L^{2}_{\xi,\eta}} \lesssim \| f \|_{L^{2}_{x,y}} \| g \|_{L^{2}_{x,y}}
\end{align*}
where we denote $\phi(\tau, \xi, \eta) = \langle \tau - m(\xi,\eta)\rangle$.
With this in mind, we observe that the left side of the inequality in Lemma \ref{lem3.3} can be estimated by Lemma \ref{200} as
\begin{align*}
\| N(u) \|_{Y^{\sigma_1, 0, -\beta}} & \lesssim \sigma \left\| \frac{\xi \langle \xi \rangle^{-1}}{\langle \phi(\tau, \xi, \eta)\rangle^{\beta}} \int \frac{e^{\sigma_1 |\xi - \xi_1|} \hat{u}(\xi - \xi_1, \eta - \eta_1, \tau - \tau_1)}{\langle \xi - \xi_1 \rangle^{-1}} \right. \times \\
& \qquad \qquad \times \left. \frac{e^{\sigma_1 |\xi_1|} \hat{u}(\xi_1, \eta_1, \tau_1)}{\langle \xi_1 \rangle^{-1}}  \ d\xi_1 d\eta_1 \right\|_{L^{2}_{\xi,\eta}}.
\end{align*}
If we apply Lemma \ref{bilinear} with $s_1 = -1$, $s_2 = 0$, it will follow from the comments above that
\[
\| N(u) \|_{Y^{\sigma_1, 0, -\beta}} \lesssim \| u \|_{Y^{\sigma_1, 0, b}}^{2}.
\]
\end{proof}

\subsection{Proof of Proposition \ref{th03}}
Begin by applying the operator $A^{\sigma_1, 0}$ to equation \eqref{p1}.  If we let $U = A^{\sigma_1, 0} u$, then equation \eqref{p1} becomes
\[
U_{t} - U_{xxxxx} + \partial_{x}^{-1} U_{yy} + U U_{x} = N(u),
\]
where $N(u)$ is as defined in Lemma \ref{lem3.3}.  Multiplying this by $U$ and integrating with respect to the spatial variables, we obtain
\[
\int U U_{t} - U U_{xxxxx} + U \partial^{-1}_{x} U_{yy} + U^2 U_{x}\ dxdy = \int U N(u)\ dxdy.
\]
If we apply integration by parts, we may rewrite the left-hand side as
\[
\frac{d}{dt} \int \frac{1}{2} U^{2}\ dxdy + \int U_{xx} U_{xxx} dx dy - \int U_{y} \partial_{x}^{-1} U_{y}\ dxdy + \int U^2 U_{x}\ dxdy,
\]
which can then be rewritten as
\begin{align*}
& \frac{1}{2} \frac{d}{dt} \int U^{2}\ dxdy + \frac{1}{2} \int \partial_{x}(U_{xx}^{2}) dx dy - \frac{1}{2} \int \partial_{x} [( \partial_{x}^{-1} U_{y})^2] \ dxdy \\ 
& \qquad \qquad + \frac{1}{3} \int \partial_{x} (U^3) \ dxdy.
\end{align*}
For $U$ and its derivatives vanishing at infinity, we thus obtain the formal identity
\[
\frac{d}{dt} \int U^{2}(x,y,t)\ dxdy = 2\int U(x,y,t) N(u)(x,y,t)\ dx dy.
\]
Integrating with respect to time yields
\begin{align*}
\int U^{2}(x,y,t)\ dxdy & = \int U^{2}(x,y,0)\ dxdy \\
& \qquad + 2\int_{0}^{t} \int U(x,y,t') \partial_{x} N(u)(x,y,t')\ dx dy dt'.
\end{align*}
Applying Cauchy-Schwarz and the definition of $U$, we obtain
\[
\| u(t) \|_{G^{\sigma_1, 0}}^{2} \leq \| f \|_{G^{\sigma_1, 0}}^{2} + \| u \|_{Y^{\sigma_1, 0, b}} \| N(u) \|_{Y^{0, 0, -\beta}(I)},
\]
where $\beta$ is as defined previously.  If we now apply Proposition \ref{th03} and the fact that $\beta < 1/2 < b$, we can further estimate this by
\begin{equation}\label{almost}
\| u(t) \|_{G^{\sigma_1, 0}}^{2} \leq \| f \|_{G^{\sigma_1, 0}}^{2} + C\sigma \| u \|_{Y^{\sigma_1, 0, b}}^{3},
\end{equation}
as desired.

\subsection{Proof of Theorem \ref{the1.3}}

With the tools established in the previous section, we may begin the proof of Theorem \ref{the1.3}.  By the embedding in equation \eqref{embed}, it suffices to consider the case $\sigma_2 = 0$.  To begin, let us first suppose that $T^{*}$ is the supremum of the set of times $T$ for which
\[
 u \in C([0,T]; G^{\sigma_1,0}).
\]
If $T^{*} = \infty$, there is nothing to prove, so let us assume that $T^{*} < \infty$.  In this case, it suffices to prove that
\begin{equation}\label{3.3.6}
u \in C\left([0, T], G^{\sigma(T), 0} \right)
\end{equation}
for some $\sigma(T) > 0$ and all $T > T^{*}$.  To show that this is the case, we will use Theorem \ref{the1.2} and Proposition \ref{th03} to construct a solution which exists over subintervals of width $\delta$, using the parameter $\sigma$ to control the growth of the norm of the solution.  Thus, the desired result will follow from the following proposition:

\begin{prop}\label{finalprop}
Let $T > 0$ and $\delta > 0$ be numbers such that $n \delta \leq T < (n + 1) \delta$.  Then the solution $u$ to the Cauchy problem \eqref{p1} satisfies
\begin{equation}\label{ineq1}
\sup_{t \in [0, n \delta]} \| u(t) \|_{G^{\sigma(T), 0}}^{2} \leq \| f \|_{G^{\sigma(T), 0}}^{2} + 2^3 C \sigma(T) n \| f \|_{G^{\sigma_1, 0}}^{3}
\end{equation}
and
\begin{equation}\label{ineq2}
\sup_{t \in [0, n \delta]} \| u(t) \|_{G^{\sigma(T), 0}}^{2} \leq 4 \| u(t) \|_{G^{\sigma_1, 0}}^{2}
\end{equation}
if
\[
\sigma(T) \leq \sigma_1 \quad \textrm{and} \quad \sigma(T) \leq \frac{C (1 + \| f \|_{G^{\sigma_1, 0}})^{1/\gamma - 1}}{T}
\]
for some constant $C > 0$.
\end{prop}
\begin{proof}
By induction on $n$.  The base case $n = 1$ follows from equation \eqref{almost}, Theorem \ref{the1.2}, and the embedding $G^{\sigma_1, \sigma_2} \hookrightarrow G^{\sigma_1', \sigma_2'}$ when $\sigma_1' \leq \sigma_1$ and $\sigma_2' \leq \sigma_2$.  Suppose, then, that the result holds for $n \leq k$.  The inductive hypothesis then tells us that
\[
\sup_{t \in [0, k \delta]} \| u(t) \|_{G^{\sigma(T), 0}} \leq \| f \|_{G^{\sigma(T), 0}}^{2} + 2^3 C \sigma(T) k \| f \|_{G^{\sigma_1, 0}}^{3}
\]
and
\[
\sup_{t \in [0, k \delta]} \| u(t) \|_{G^{\sigma(T), 0}}^{2} \leq 4 \| f \|_{G^{\sigma_1, 0}}^{2}
\]

If we apply the inductive hypothesis on the interval $[k\delta, (k+1) \delta]$, then
\[
\sup_{t \in [k \delta, (k+1) \delta]} \| u(t) \|_{G^{\sigma(T), 0}} \leq \| u(k\delta) \|_{G^{\sigma(T), 0}}^{2} + 2^3 C \sigma(T) \| u(k\delta) \|_{G^{\sigma_1, 0}}^{3}
\]
and
\[
\sup_{t \in [k \delta, (k + 1) \delta]} \| u(t) \|_{G^{\sigma(T), 0}}^{2} \leq 4 \| u(k\delta) \|_{G^{\sigma_1, 0}}^{2}.
\]
If we apply equations \eqref{ineq1} and \eqref{ineq2} to these, we get
\[
\| u(k \delta) \|_{G^{\sigma(T), 0}} \leq \| f \|_{G^{\sigma(T), 0}}^{2} + 2^3 C k \sigma(T) \| u(k\delta) \|_{G^{\sigma_1, 0}}^{3}
\]
and
\[
\| u(k \delta) \|_{G^{\sigma_1, 0}}^2 \leq 4 \| f \|_{G^{\sigma_1, 0}}^2.
\]
Combining these together, we obtain
\begin{align*}
\sup \| u(t) \|_{G^{\sigma(T),0}} & \leq \left( \| f \|_{G^{\sigma(T), 0}}^{2} + 2^3 C k \sigma(T) \| f \|_{G^{\sigma_1, 0}}^{3}\right) + 2^3 C \sigma(T) \| f \|_{G^{\sigma_1, 0}}^{3} \\
& = \| f \|_{G^{\sigma(T), 0}}^{2} + 2^3 C (k+1) \sigma(T) \| f \|_{G^{\sigma_1, 0}}^{3}.
\end{align*}
It follows that
\[
\sup_{t \in [0, (k+1) \delta]} \| u(t) \|_{G^{\sigma(T),0}} \leq \| f \|_{G^{\sigma(T), 0}}^{2} + 2^3 C (k+1) \sigma(T) \| f \|_{G^{\sigma_1, 0}}^{3}.
\]
To complete the proof, we need to show that
\begin{equation}\label{complete}
\sup_{t \in [0, (k+1) \delta]} \| u(t) \|_{G^{\sigma(t),0}}^2 \leq 4 \| f \|_{G^{\sigma_1, 0}}^2.
\end{equation}
Next, we observe that the assumption $n \delta \leq T < (n+1)\delta$ implies that
\[
n \leq \frac{T}{\delta} < n + 1 \leq \frac{T}{\delta} + 1 \leq \frac{2T}{\delta}.
\]
Since $k + 1 \leq n + 1$, we have
\begin{align*}
2^3 C (k+1) \sigma(T) \| f \|_{G^{\sigma_1, 0}} & \leq 2^4 C \frac{2T}{\delta} \sigma(T) \| f \|_{G^{\sigma_1, 0}} \\
& \leq \frac{C \| f \|_{G^{\sigma_1, 0}}}{(1 + \| f \|_{G^{\sigma_1, 0}})^{1/\gamma}} T \sigma(T).
\end{align*}
Recalling that $1/\gamma > 1$, we thus have
\[
2^3 C (k+1) \sigma(T) \| f \|_{G^{\sigma_1, 0}} \leq \frac{CT}{(1 + \| f \|_{G^{\sigma_1, 0}})^{1/\gamma - 1}} \sigma(T).
\]
Thus, for equation \eqref{complete} to hold, it suffices to have
\[
\sigma(T) \leq \frac{C (1 + \| f \|_{G^{\sigma_1, 0}})^{1/\gamma - 1}}{T}.
\]
Since we have shown that the result holds for $n = 1$, and we have shown that the result for $n = k$ implies it for $n = k+1$, then the result holds for all $n$.  This completes the proof of Proposition \ref{finalprop}, from which Theorem \ref{the1.3} follows as an immediate corollary.

\end{proof}

\bibliographystyle{siam}
\bibliography{gg}

\end{document}